\documentclass[10pt]{amsart}
\usepackage{amssymb}

\usepackage{graphicx}
\usepackage{subfigure}

\newtheorem{theorem}{Theorem}[section]
\newtheorem{lemma}[theorem]{Lemma}
\newtheorem{corollary}[theorem]{Corollary}
\newtheorem{proposition}[theorem]{Proposition}

\theoremstyle{definition}
\newtheorem{definition}[theorem]{Definition}

\newtheorem{question}[theorem]{Question}

\numberwithin{equation}{section}
\numberwithin{figure}{section}

\newcommand{\Z}{\mathbb{Z}}
\newcommand{\R}{\mathbb{R}}
\newcommand{\Q}{\mathbb{Q}}

\newcommand{\cm}{\cdot}

\begin{document}

\title{Knot Concordance and Homology Cobordism}


\author[Cochran]{Tim D. Cochran}
\address{Department of Mathematics\\Rice University\\P.O. Box 1892\\Houston, TX 77251}
\curraddr{}
\email{cochran@math.rice.edu}
\thanks{Partially supported by National Science Foundation  DMS-1006908}

\author[Franklin]{Bridget D. Franklin}
\address{Department of Mathematics\\Rice University\\P.O. Box 1892\\Houston, TX 77251}
\curraddr{}
\email{bridget.franklin@rice.edu}
\thanks{Partially supported by Nettie S. Autry Fellowship}

\author[Hedden]{Matthew Hedden}
\address{Department of
Mathematics\\Michigan State University\\East Lansing, MI 48824} 
\email{mhedden@math.msu.edu}
\thanks{Partially supported by NSF DMS-0906258}

\author[Horn]{Peter D. Horn}
\address{Department of Mathematics\\Columbia University\\2990 Broadway\\New York, NY 10027}
\curraddr{}
\email{pdhorn@math.columbia.edu}
\thanks{Partially supported by NSF Postdoctoral Fellowship DMS-0902786}


\date{19 February, 2011}

\dedicatory{}

\commby{Daniel Ruberman}

\begin{abstract} We consider the question: ``If the zero-framed surgeries on two oriented knots in $S^3$ are $\mathbb{Z}$-homology cobordant, preserving the homology class of the positive meridians, are the knots themselves concordant?''  We show that this question has a negative answer in the smooth category, even for topologically slice knots. To show this we first prove that  the zero-framed surgery on $K$ is $\Z$-homology cobordant to the zero-framed surgery on many of its winding number one  satellites $P(K)$. Then we prove that in many cases the $\tau$ and $s$-invariants of $K$ and $P(K)$ differ. Consequently neither $\tau$ nor $s$ is an invariant of the smooth homology cobordism class of the zero-framed surgery. We also show, that a natural rational version of this question has a negative answer in both the topological and smooth categories, by proving similar results for $K$ and its $(p,1)$-cables. 
\end{abstract}

\maketitle


\section{Introduction}\label{sec:intro}

A \emph{knot} $K$ is the image of a tame embedding of an oriented circle into $S^3$. Two knots, $K_0\hookrightarrow S^3\times \{0\}$ and $K_1\hookrightarrow S^3\times \{1\}$, are \emph{CAT concordant} (CAT= smooth or topological locally flat) if there exists a proper CAT embedding of an annulus into $S^3\times [0,1]$ that restricts to the knots on $S^3\times \{0,1\}$. A \emph{CAT slice knot} is one that is the boundary of a CAT embedding of a $2$-disk in $B^4$. Recall that to any knot is associated a closed $3$-manifold, $M_K$, called the \emph{zero-framed surgery on $K$}, which is obtained from $S^3$ by removing a solid torus neighborhood of $K$ and replacing it with another solid torus in such a way that the longitude of $K$ bounds the meridian of the new solid torus. The following purely $3$-dimensional question has been widely studied: to what extent does the  homeomorphism type of $M_K$ (or any other fixed Dehn surgery) determine the knot type of $K$? The questions we consider are the analogous $4$-dimensional questions. Specifically, the question of whether two knots, $K_0$ and $K_1$, are CAT concordant is closely related to the question of whether or not the $3$-manifolds, $M_{K_0}$ and $M_{K_1}$, are CAT homology cobordant. For if the knots are concordant via a CAT embedded annulus $A$, then by Alexander duality, the exterior of $A$ is a homology cobordism relative boundary between the exteriors of the knots. Moreover the positively-oriented meridians of the knots are isotopic in this exterior. It follows, by adjoining $S^1\times D^2\times [0,1]$ (i.e. one does ``zero-framed surgery on the annulus''), that $M_{K_0}$ and $M_{K_1}$ are homology cobordant \emph{rel meridians}. In this paper the latter will mean that  the positively-oriented meridians of the knots represent the same class in the first homology group of the cobordism. The converse is a long-standing open question:

\begin{question}\label{question:Kirby} If $M_{K_0}$ and $M_{K_1}$ are homology cobordant rel meridians (if necessary also assume that $\pi_1$ of the cobordism is normally generated by either meridian), then are $K_0$ and $K_1$ concordant?  (See Kirby's Problem List ~\cite[Problem $1.19$ and remarks]{Kproblemlist}).
\end{question}

The original question, without the meridian-preserving condition, was shown to have a negative answer by Livingston, who gave examples in ~\cite{Li7} of knots that are not concordant to their own reverses (change the orientation of the circle).

Evidence for a positive answer to Question~\ref{question:Kirby} is provided by the following well-known result.

\begin{proposition}\label{prop:basichomcobordism}  Suppose $K$ is a knot and $U$ is the trivial knot. Then the following are equivalent:
\begin{itemize}
\item [a.] $M_K$ is CAT homology cobordant rel meridians  to $M_U$  via a cobordism $V$ whose $\pi_1$ is normally generated by a meridian of $K$.
\item [b.] $M_K=\partial W$, where $W$ is a CAT manifold  that is a homology circle and whose $\pi_1$ is normally generated by the meridian.
\item [c.] $K$ bounds a CAT embedded $2$-disk in a CAT manifold $\mathcal{B}$ that is homeomorphic to $B^4$.
\item [d.] $K$ is CAT concordant to $U$ in a CAT $4$-manifold that is homeomorphic to $S^3\times [0,1]$.
\end{itemize}
Moreover $d \Rightarrow a$ even if $U$ is not assumed to be the trivial knot.
\end{proposition}

In particular the case $a \Rightarrow d$ of Proposition~\ref{prop:basichomcobordism} in the TOP category yields:
\begin{corollary}\label{cor:truTOP} Question~\ref{question:Kirby} has a positive answer in the topological category if one of the knots is the trivial knot. The same is true in the smooth category if $B^4$ has a unique smooth structure up to diffeomorphism.
\end{corollary}

\begin{proof}[Proof of Proposition~\ref{prop:basichomcobordism}] The implication $d\Rightarrow a$ has already been explained above (and is true for any two knots). Note that $M_U\cong S^1\times S^2$, which is the boundary of the homology circle $S^1\times B^3$. Adjoining the latter to one end of the homology bordism $V$, provided by $a$, yields $W$ which shows $a\Rightarrow b$. Given $W$ from $b$, add a zero-framed $2$-handle to $\partial W$ along the meridian of $K$. The resulting CAT manifold $\mathcal{B}$ is a simply-connected homology $4$-ball whose boundary is $S^3$. By Freedman's theorem, $\mathcal{B}$ is homeomorphic to $B^4$ ~\cite{Fr1}. The core of the attached $2$-handle is a flat disk whose boundary is a copy of $K\hookrightarrow S^3$. Thus $b\Rightarrow c$. To show $c\Rightarrow d$, merely remove a small $4$-ball from $\mathcal{B}$ centered at a point of the $2$-disk.
\end{proof}

Motivation for Question~\ref{question:Kirby} was provided by the following observation, whose proof is similar to the proof of $a \Rightarrow d$ above. If the \textit{exteriors} of the knots $K_0$ and $K_1$ are CAT-homology cobordant relative boundary via a cobordism whose $\pi_1$ is normally-generated by a meridian, then $K_0$ and $K_1$ are CAT-concordant in $S^3\times [0,1]$ equipped, in the smooth case, with a possibly exotic smooth structure. Thus (modulo the smooth $4$-dimensional Poincar\'{e} conjecture) concordance truly is equivalent to  homology cobordism of the knot exterior (with the added $\pi_1$ condition).

Nonetheless, in this paper we show that Question~\ref{question:Kirby} has a \emph{negative} answer in the smooth category, even for topologically slice knots.  We do not resolve it in the topological category. To accomplish this we first prove in Section~\ref{sec:mainresult}  that 

\newtheorem*{cor:windingonesatelliteresult}{Corollary~\ref{cor:windingonesatelliteresult}}
\begin{cor:windingonesatelliteresult}   Suppose $P$ is a knot embedded in a solid torus with winding number $1$  such that $P$ is unknotted when the solid torus is embedded in $S^3$ in the standard unknotted fashion. Then, for any knot $K$,  the zero-framed surgery on the satellite knot $P(K)$  is smoothly $\Z$-homology cobordant rel meridians to zero-framed surgery on $K$.
\end{cor:windingonesatelliteresult}

Then in Section~\ref{sec:Zhomologycobordism} we give an example of a $P$ and a topologically slice knot $K$ such that $K$ and $P(K)$ are  \textit{not} smoothly concordant. Combining these results we have a negative answer to Question~\ref{question:Kirby} in the smooth category.

\newtheorem*{thm:smooth} {Theorem~\ref{thm:smooth}}\begin{thm:smooth} There exist topologically slice knots whose  zero surgeries are smoothly $\Z$-homology cobordant rel meridians, but which are not smoothly concordant.
\end{thm:smooth}

Moreover,  our examples can be distinguished by the smooth concordance invariants $\tau$ and $s$ \cite{OzSz2,RasThesis,RasSlice},  so:

\newtheorem*{cor:difftau}{Corollary~\ref{cor:difftau}} \begin{cor:difftau} The smooth knot concordance invariants $\tau$ and $s$ are not invariants of the smooth homology cobordism class of $M_K$.
\end{cor:difftau}  

Proposition~\ref{prop:basichomcobordism} has an analogue for \emph{rational} homology cobordism. In particular it is related to notions of \emph{rational concordance} which have been  previously studied \cite{CO1,ChaKo1,ChaKo2,Gi6, Ka4} and treated systematically by Cha in ~\cite{Cha2}. Before stating this analogue, we review some terminology.

Suppose that $R\subset\Q$ is a non-zero subring. Recall that a space $X$ \emph{is an $R$-homology $Y$} means that $H_*(X;R)\cong H_*(Y;R)$.  Knots $K_0$ and $K_1$ in $S^3$ are said to be \emph{CAT $R$-concordant} if there exists a compact, oriented CAT $4$-manifold $W$, that is an $R$-homology $S^3\times [0,1]$, whose boundary is $S^3\times \{0\}\sqcup -(S^3\times \{1\})$, and in which there exists a properly CAT embedded annulus $A$ which restricts on its boundary to the given knots. We then say that $K_0$ is \emph{CAT $R$-concordant} to $K_1$ \emph{in $W$}. We say that $K$ is \emph{CAT $R$-slice} if it is CAT $R$-concordant to $U$, or equivalently if it bounds a CAT embedded $2$-disk in an $R$-homology $4$-ball whose boundary is $S^3$. The latter notion agrees with ~\cite{ChaKo1,Cha2} but is what Kawauchi calls weakly $\Q$-slice \cite{Ka4}.

\begin{definition}\label{def:relmer} $M_{K_0}$ is $R$\emph{-homology cobordant  rel meridians} to $M_{K_1}$  if these two $3$-manifolds are $R$-homology cobordant via a $4$-manifold $V$ such that in $H_1(V;R)$ the positively-oriented meridians differ by a positive unit of $R$.
\end{definition}

\begin{proposition}\label{prop:Rbasichomcobordism}  Suppose $K$ is a knot in $S^3$ and $R\subset\Q$ is a non-zero subring. Then the following are equivalent:
\begin{itemize}
\item [a.] $M_K$ is CAT $R$-homology cobordant rel meridians to $M_U$.
\item [b.] $M_K=\partial W$, where $W$ is a CAT manifold  that is an $R$-homology circle.
\item [c.] $K$ is CAT $R$-slice.
\item [d.] $K$ is CAT $R$-concordant to $U$.
\end{itemize}
Moreover $d \Rightarrow a$ for any two knots.
\end{proposition}

The proof of Proposition~\ref{prop:Rbasichomcobordism} is essentially identical to that of Proposition~\ref{prop:basichomcobordism}. Proposition~\ref{prop:Rbasichomcobordism} suggests that invariants that obstruct knots from being rationally concordant might be dependent only on the rational homology cobordism class of the zero-framed surgery. Further evidence for this is  provided by the following observation,  which follows immediately from $d \Rightarrow c$ of Proposition~\ref{prop:Rbasichomcobordism} and Ozsv\'{a}th-Szab\'{o}~\cite[Theorem 1.1]{OzSz2}.

\begin{corollary}\label{cor:tauratU} For any $R$, if $M_K$ is smoothly $R$-homology cobordant to $M_U$, then $\tau(K)=0$.
\end{corollary}

This suggests the following ``rational version'' of Question~\ref{question:Kirby}, which is the question of whether or not  $a \Rightarrow d$ of Proposition~\ref{prop:Rbasichomcobordism} holds for any two knots (when $R=\Q$).

\begin{question}\label{question:KirbyQ} If $M_{K_0}$ and $M_{K_1}$ are CAT $\Q$-homology cobordant rel meridians,  then are $K_0$ and $K_1$ CAT $\Q$-concordant?
\end{question}

\noindent Of course by Proposition~\ref{prop:Rbasichomcobordism} this question has a positive answer if one of the knots is the unknot.

We show that Question~\ref{question:KirbyQ} has a \emph{negative} answer in the smooth category by showing that the examples from Theorem~\ref{thm:smooth} are not even smoothly $\Q$-concordant, because the $\tau$-invariant is an invariant of smooth $\Q$-concordance (see Proposition~\ref{prop:tau1}). We go on to show that Question~\ref{question:KirbyQ} has a \emph{negative} answer even in the topological category.  To accomplish this we first prove in Section~\ref{sec:mainresult} :

\newtheorem*{cor:cable}{Corollary~\ref{cor:cable}}
\begin{cor:cable} For any knot $K$ and any positive integer $p$, zero-framed surgery on $K$ is smoothly $\Z\left[\frac{1}{p}\right]$-homology cobordant rel meridians to zero-framed surgery on the $(p,1)$-cable of $K$.
\end{cor:cable}

\noindent Then in Section~\ref{sec:Qinvts} we observe that there are elementary classical invariants that obstruct a knot's being topologically rationally concordant to its $(p,1)$-cable. Even among topologically slice knots, the $\tau$ invariant can be used to obstruct $K$  being smoothly rationally concordant to its $(p,1)$-cable. 

Thus in summary we show:
\newtheorem*{thm:mainapplication}{Theorem~\ref{thm:mainapplication}}
\begin{thm:mainapplication} The answer to Question~\ref{question:KirbyQ} is ``No,'' in both the smooth and topological categories. In the smooth category there exist counterexamples that are topologically slice.
\end{thm:mainapplication}

We remark in passing that the analogues of Proposition~\ref{prop:basichomcobordism} and Corollary~\ref{cor:truTOP}  hold for links ~\cite[Theorem 2, p.19]{Hi} and have been the basis of recent attempts to resolve the smooth $4$-dimensional Poincar\'{e} Conjecture ~\cite{FGMW}.  However,  there are examples of links with whose  zero surgeries are \emph{diffeomorphic}  but whose concordance classes are distinguished by elementary invariants such as higher-order linking numbers ~\cite[Figure 4.7]{C3}. Even more simply, consider the link, $L_m$,  whose first component is an $m$-twist knot and whose second component is a trivial circle linking the $m$-twisted band. The zero-framed surgery on $L_m$ is then independent of $m$ but the concordance type of the first component is dependent on $m$.  See  ~\cite[Fig.1]{CO1a}~\cite{CO1} for different examples.

\bigskip
\noindent{\bf{Acknowledgment:}} Thanks to Paul Kirk, Chuck Livingston, and Danny Ruberman for their interest and some useful suggestions in Section~\ref{sec:Zhomologycobordism}.

\section{Homology cobordism and satellite knots}\label{sec:mainresult}

We recall the notation for a satellite construction ~\cite[p. 10]{Lick2}. Suppose $P$ is an oriented knot  in the  solid torus $ST$, called a \textit{pattern knot}. An example is shown in Figure~\ref{fig:patterns1}~(a). For any knot type $K$ in $S^3$ we denote by $P(K)$ the satellite of  $K$ obtained by using $P$ as a pattern.  The \textit{winding number} of $P$ is the algebraic intersection number of $P$ with a meridional disk of $ST$. Let $\widetilde{P}=P(U)$ be the knot in $S^3$ obtained by viewing ST as the standard unknotted solid torus in $S^3$. For example, for the $P$ shown in Figure~\ref{fig:patterns1}~(a), $\widetilde{P}$ is the trivial knot. 

\begin{theorem}\label{thm:mainsatelliteresult}  Suppose $P$ is a pattern with positive winding number $p$ such that $\widetilde{P}$ is CAT $\mathbb{Z}\left[\frac{1}{p}\right]$ slice. Then, for any knot $K$,  the zero-framed surgery on $K$ is CAT $\Z\left[\frac{1}{p}\right]$-homology cobordant rel meridians to zero-framed surgery on the satellite knot $P(K)$.
\end{theorem}

In the special case that $\widetilde{P}$ is unknotted and the winding number is $1$ we can set $p=1$ and get:

\begin{corollary}\label{cor:windingonesatelliteresult}  Suppose $P$ is a pattern with  winding number $1$ such that $\widetilde{P}$ is unknotted. Then, for any knot $K$,  the zero-framed surgery on $K$ is smoothly $\Z$-homology cobordant rel meridians to zero-framed surgery on the satellite knot $P(K)$.
\end{corollary}

As another special case take the pattern knot to be a standard circle in $ST$ with winding number $p$ (so that $\widetilde{P}$ is unknotted). Then $P(K)$ is called the $(p,1)$-cable of $K$. In this case we have:

\begin{corollary}\label{cor:cable} For any knot $K$ and any positive integer $p$, zero-framed surgery on $K$ is smoothly $\Z\left[\frac{1}{p}\right]$-homology cobordant rel meridians to zero-framed surgery on the $(p,1)$-cable of $K$.
\end{corollary}

Before giving a proof of Theorem~\ref{thm:mainsatelliteresult}, we offer an alternative ``Kirby calculus" proof in the case of greatest interest, namely when $\widetilde{P}$ is unknotted. 

\begin{proof}[Proof in case $\widetilde{P}=U$] We describe a cobordism, $V$, from $M_K$ to $M_{P(K)}$. Let $X$ be the $4$-manifold obtained by adding a one-handle to $M_K\times [0,1]$. Thus $\partial^+X$ is $M_K\# S^1\times S^2$. This $3$-manifold is shown (schematically) in Figure~\ref{fig:Tim1}  (a), where we use that adding a $0$-framed $2$-handle along the unknotted circle $\widetilde{P}$ also yields $M_K\# S^1\times S^2$. Now add a $0$-framed $2$-handle, $H$, along the circle shown in Figure~\ref{fig:Tim1}  (b), to arrive at the desired cobordism $V$. Since this $2$-handle equates the meridian of $K$ with $p$ times the meridian of $\widetilde{P}$, $V$ is a $\Z\left[\frac{1}{p}\right]$-homology cobordism between $M_K$ and $\partial^+V$.

\begin{figure}[!h]
	    \begin{center}
	    	\begin{picture}(200,84)(0,0)
	    		\subfigure[$S^1\times S^2\# M_K$]
			    {
			        \includegraphics{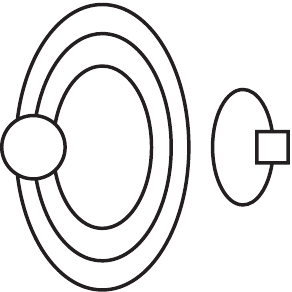}
					\put(-79,39){$P$}
					\put(-9,38){\small{K}}
					\put(-85,10){$0$}
					\put(-12,15){$0$}
					
			    }
				\hspace{.5cm}
			    \subfigure[$\partial ^+V$]
			    {
			        \includegraphics{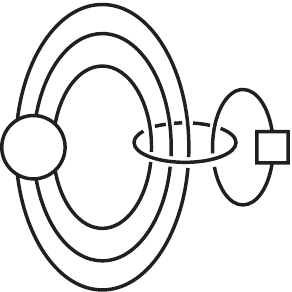}
					\put(-79,39){$P$}
					\put(-85,10){$0$}
					\put(-9,38){\small{K}}
					\put(-12,15){$0$}
					\put(-55,39){$0$}
			    }

	    	\end{picture}
		    \caption{}
			\label{fig:Tim1}
	    \end{center}
	\end{figure}
	
	\begin{figure}[!h]
	    \begin{center}
	    	\begin{picture}(200,84)(0,0)
				\subfigure[$\partial^+V$]
				{
					\includegraphics{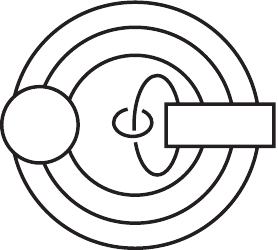}
					\put(-73,34){$P$}
					\put(-22,33){\small{K}}
					\put(-52,41){$0$}
					\put(-47,23){$0$}
					\put(-81,10){$0$}
				}
				\hspace{.5cm}
				\subfigure[$M_{P(K)}$]
				{
					\includegraphics{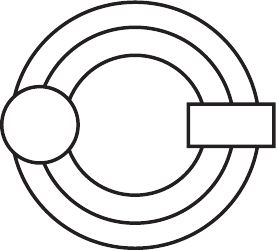}
					\put(-73,34){$P$}
					\put(-18,33){\small{K}}
					\put(-81,10){$0$}
				}
	    	\end{picture}
		    \caption{}
			\label{fig:Tim2}
	    \end{center}
	\end{figure}
\noindent To see that $\partial^+V\cong M_{P(K)}$, start from the picture of $\partial^+V$ in Figure~\ref{fig:Tim1} (b). Slide all strands of $\widetilde{P}$ that pass through the attaching circle of $H$ over the right-most $2$-handle, arriving at Figure~\ref{fig:Tim2} (a). Now the small linking circle in Figure~\ref{fig:Tim2} (a) can be ``cancelled'' with its linking partner resulting in the picture of $M_{P(K)}$ shown in Figure~\ref{fig:Tim2} (b). In the first homology of $V$, the meridian of $K$ corresponds to $p$ times the meridian of $P(K)$, so $V$ is a
 $\Z\left[\frac{1}{p}\right]$-homology cobordism rel meridians.
\end{proof}

\begin{proof}[General Proof of Theorem~\ref{thm:mainsatelliteresult}]  Let $\alpha$ denote a meridian of ST so that $\ell k(\alpha,\widetilde{P})=p$.  Note that $S^3\setminus \nu(\alpha)$  can be identified with $ST$. Since $P(K)$ is a satellite knot with  pattern $P$ and companion $K$, the exterior of $P(K)$ decomposes into two pieces, $S^3\setminus \nu( K)$ and  $ST\setminus P$, and the latter may be identified with $S^3\setminus \nu(\widetilde{P})\setminus \nu(\alpha)$.  Thus the $0$-framed Dehn filling, $M_{P(K)}$, decomposes into $S^3\setminus \nu( K)$ and $M_{\widetilde{P}}\setminus \nu(\alpha)$, identifying  the boundary of $\nu(\alpha)$ with the boundary of $\nu(K)$. Additionally recall that $M_K$ decomposes into $S^3\setminus \nu(K)$ and a ``surgery'' solid torus.

With these facts in mind, we now construct  a smooth $4$-manifold $E$ whose boundary is the disjoint union $M_K \sqcup M_{\widetilde{P}} \sqcup -M_{P(K)}$. Begin with the disjoint union of $M_K  \times [0, 1]$ and $M_{\widetilde{P}} \times [0,1]$. Then identify the solid torus $\nu(\alpha) \times \{1\}$ in $M_{\widetilde{P}} \times \{1\}$ with the surgery solid torus  in  $M_K  \times \{1\}$. Do this in such a way that (a parallel push-off of) $\alpha$ is identified to a meridian of $K$. Then the third boundary component of $E$ is
$$
\left(M_{\widetilde{P}}\setminus \nu(\alpha)\right)\cup S^3\setminus \nu(K)\equiv M_{P(K)},
$$
using the first paragraph. Hence the boundary of $E$ is as claimed. Furthermore, note that under the inclusion maps on first homology,
\begin{equation}\label{eq:h1facts}
[\mu_K\times \{0\}]=[\mu_K\times \{1\}]= [\alpha\times \{1\}]= [\alpha\times \{0\}]= p\,[\mu_{\widetilde{P}}], ~~\text{and}
\end{equation}
\begin{equation}\label{eq:moreh1facts}
[\mu_{\widetilde{P}}]= [\mu_{P(K)}].
\end{equation}
We may analyze the homology of $E$ by the Mayer-Vietoris sequence with $\Z$ coefficients below:
$$
0\to H_2(M_{\widetilde{P}})\oplus H_2(M_K)\to H_2(E)\to H_1(\nu(\alpha))\overset{\psi}{\rightarrow} H_1(M_{\widetilde{P}})\oplus H_1(M_K)\to
$$
$$
\hspace{108mm}H_1(E)\to 0
$$
By ~\eqref{eq:h1facts}, $\psi([\alpha])=(p[\mu_{\widetilde{P}}],[\mu_K])$. This fact, together with  ~\eqref{eq:moreh1facts}, yield the first two claims of the following lemma, which may be compared with \cite[Lemma 2.5]{CHL3}. Since  $\psi$ is injective and  $H_2(\nu(\alpha))=0$, the third claim of the lemma follows.

\begin{lemma}\label{lem:Efacts}
The inclusion maps induce the following
\begin{enumerate}	
	\item an isomorphism $H_1\left(M_{\widetilde{P}};\mathbb{Z}\right)\cong H_1\left(E;\mathbb{Z}\right)=~\langle[\mu_{\widetilde{P}}]\rangle=~\langle[\mu_{P(K)}]\rangle$;
\item an isomorphism $H_1\left(M_{K};\mathbb{Z}\left[\frac{1}{p}\right]\right)\rightarrow H_1\left(E;\mathbb{Z}\left[\frac{1}{p}\right]\right)$;
	\item an isomorphism $H_2(E;\mathbb{Z}) \cong H_2\left(M_{\widetilde{P}};\mathbb{Z}\right)\oplus H_2(M_{K};\mathbb{Z})$
		\end{enumerate}
	\end{lemma}

Finally, since $\widetilde{P}$ is CAT $\Z[1/p]$-slice, by Proposition~\ref{prop:Rbasichomcobordism}, $M_{\widetilde{P}}=\partial W$ where $W$ is a CAT $\Z[1/p]$-homology $S^1$  whose first homology is generated by $\mu_{\widetilde{P}}$. Let $V$ be the CAT $4$-manifold obtained by attaching $W$ to $E$ along $M_{\widetilde{P}}$.  Note $V$ has boundary $M_K \sqcup -M_{P(K)}$.

We claim that $V$ is the desired CAT $\Z[1/p]$-homology cobordism	rel meridians between $M_K$ and $M_{P(K)}$. The Mayer-Vietoris sequence with  $\Z[1/p]$-coefficients for $V=E\cup W$ gives
$$
H_2(V;\Z[1/p]) \cong \frac{H_2(E;\Z[1/p])}{H_2(M_{\widetilde{P}};\Z[1/p])}\cong H_2(M_K;\Z[1/p]),
$$
where we use ($3$) of Lemma~\ref{lem:Efacts} for the last equivalence; and
$$
H_1(V;\Z[1/p]) \cong H_1(E;\Z[1/p])\cong H_1(M_K;\Z[1/p])=~<[\mu_{P(K)}]>,
$$
where we have use ($1$) and ($2$) of Lemma~\ref{lem:Efacts} for the last two equivalences.
Combining these facts, we conclude that
$$
H_2\left(V,M_K;\mathbb{Z}\left[1/p\right]\right)\cong 0\cong H_1\left(V,M_K;\mathbb{Z}\left[1/p\right]\right).
$$
Moreover,
$$
H_3\left(V,M_K;\mathbb{Z}\left[1/p\right]\right)\cong H^1(V,M_{P(K)};\mathbb{Z}\left[1/p\right])\cong 0,
$$
since $H_1\left(V,M_{P(K)};\mathbb{Z}\left[1/p\right]\right)\cong 0$. It the follows from Poincar\'{e} duality that $W$ is a $\mathbb{Z}[1/p]$-homology cobordism between $M_K$ and $M_{P(K)}$. Moreover, by ~\ref{eq:h1facts} and ~\ref{eq:moreh1facts}, it is rel meridians since $p$ is a positive unit in $\Z[1/p]$.
\end{proof}

\section{Non-concordant knots whose zero surgeries are $\mathbb{Z}$-homology cobordant}\label{sec:Zhomologycobordism}

\begin{theorem} \label{thm:smooth} There exist topologically slice knots whose  zero surgeries are smoothly $\Z$-homology cobordant rel meridians, but which are not smoothly concordant.
\end{theorem}
\begin{proof}
Let $P$ be the pattern knot in the solid torus shown in Figure~\ref{fig:patterns1}~(a), and denote by $P(K)$ the satellite of a knot $K$ obtained by using $P$ as a pattern.  

\begin{figure}[!ht]
    \begin{center}
    		\subfigure[The pattern $P$]
		    {
		        \begin{picture}(160,190)(0,-30)
					\includegraphics[bb=0 0 200 175,scale=.8]{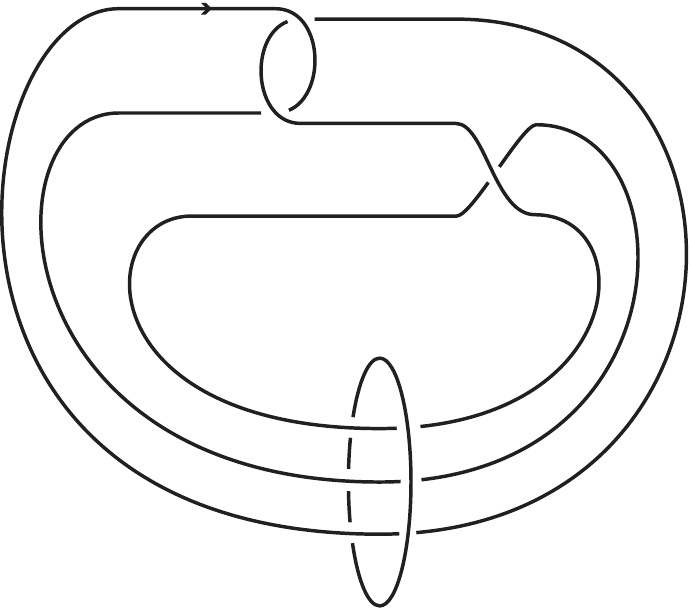}
			        \label{fig:patterninsolidtorus}
		        \end{picture}
		    }\hspace{.8cm}
		    \subfigure[A Legendrian front of $P$]
		    {
				\begin{picture}(150,190)(0,-40)
					\put(21,-20){writhe = $3$}
					\put(90,-20){$\#$ cusps = $2$}
					\put(15,-35){$\#\downarrow$ left = $0$}
					\put(85,-35){$\#\uparrow$ right = $0$}
			        \includegraphics[bb= 0 0 153 150]{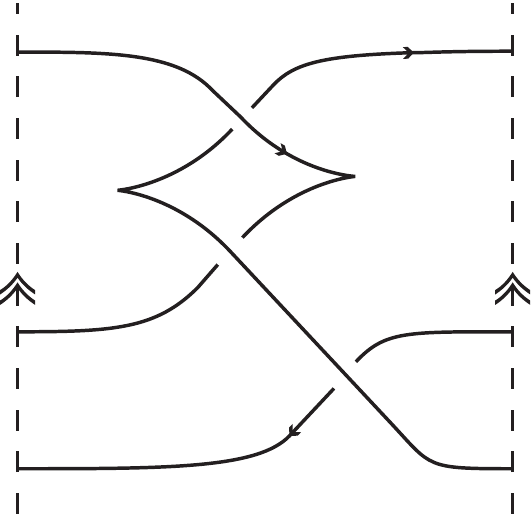}
			        \label{fig:legendrianpattern}
				\end{picture}
		    }
		\caption{}
	    \label{fig:patterns1}
    \end{center}
\end{figure}

Since the winding number of $P$ is $w=1$,  Corollary~\ref{cor:windingonesatelliteresult} shows that $M_K$ and $M_{P(K)}$ are smoothly $\Z$-homology cobordant rel meridians.   We will show that many $K$ are not smoothly concordant to $P(K)$.

An easy way to produce such $K$  is afforded by the slice-Bennequin inequality  \cite{Rudolph1997,Rudolph1995}~\cite[p.133]{EtnyreKnots}.  This inequality bounds the smooth 4-genus of a knot in terms of invariants of its Legendrian realizations in the standard contact structure on the 3-sphere.  Recall that a knot $K$ is {\em Legendrian} if it is tangent to the 2-planes of a contact structure (here we think of $K$ as a specific embedding).  Such knots have two homotopy theoretic invariants, the Thurston-Bennequin and rotation numbers, denoted $tb(K)$ and $rot(K)$ respectively.  Rather than reviewing these concepts here, we refer the reader to any number of excellent introductory sources  \cite{EtnyreKnots,GS, OzbagciStipsicz}. The slice-Bennequin inequality states
$$tb(K)+ |rot(K)|\le 2g_4(K)-1.$$
Our examples are produced as follows:  take any non-trivial knot which admits a Legendrian realization satisfying the {\em equality}:
\begin{equation} \label{eq:sharp} tb(K)=2g(K)-1,\end{equation}
where $g(K)$ denotes the 3-genus. Together with the slice-Bennequin inequality, this implies  $g_4(K)=g(K)$ .  Such knots are easy to find, e.g.\ any knot which can be represented as the closure of a positive braid.  Topologically slice examples also exist. For instance, the positive untwisted Whitehead double of any knot satisfying \eqref{eq:sharp} has a Legendrian realization satisfying \eqref{eq:sharp} \cite{Rudolph1995}\cite[Figure 9]{AR}.    We claim that for such knots one can produce a Legendrian realization of $P(K)$ satisfying $$tb(P(K))+|rot(P(K))|=2g(K)+1.$$  Granting this, the slice-Bennequin inequality implies $g_4(P(K))\ge g_4(K)+1$, thus proving the theorem.

To construct the desired Legendrian realization of $P(K)$, we first  produce a Legendrian realization of the pattern knot within the solid torus, satisfying
$$ tb(P)= 2 \ \ , \ \ rot(P)=0.$$
Figure~\ref{fig:patterns1}~(b) indicates how to do this. In the figure the dashed vertical lines are identified, thereby producing a projection of a Legendrian knot in a solid torus endowed with its standard contact structure, obtained as the quotient of $\R\times D^2\subset \R^3\subset S^3$, by translation.  Given a Legendrian pattern $P$ and a Legendrian companion $K$ one can employ the {\em Legendrian satellite construction} to produce a Legendrian satellite $P(K)$ (see \cite{NOT} or \cite[Appendix]{NgTraynor} for details).  One should  note that the framing used to identify the neighborhood of $K$ with $S^1\times D^2$ in this construction is the Legendrian framing, given by $tb(K)\in \Z$.

 Since $tb(K)=2g(K)-1>0$, we can repeatedly stabilize $K$  to achieve a new Legendrian knot satisfying $tb(K)=0$ and $rot(K)=2g(K)-1.$ Hence the framing used for the Legendrian satellite can be assumed to agree with the Seifert framing used in the standard satellite construction.  The claim is now proved by the following  formulas, which relate the invariants of the knots involved in the construction \cite[Remark 2.4]{Ng2001}:
$$ tb(P(K))=w^2 \cm tb(K) + tb(P),$$
$$ rot(P(K))= w \cm rot(K) + rot(P).$$
Here $w=1$, and we have   
$$tb(P(K))=tb(K)+tb(P)=0+2 \ \ , \ \ rot(P(K))= rot(K) + rot(P)= 2g(K)-1 + 0.$$
This proves the claim.
\end{proof}

Note that we described the Legendrian satellite construction used in the proof somewhat abstractly.  A concrete way to understand Legendrian knots in the standard contact structure on $S^3$ is through their  {\em front projections}.  Given front projections for $K$ and a pattern $P$  it is straightforward to produce a front projection for $P(K)$.  One can then explicitly compare the Thurston-Bennequin and rotation numbers using the formulas $$tb(K)=\text{writhe}(D)-\frac{1}{2}\# \text{cusps}(D),$$ $$rot(K)= \#\text{downward moving left-cusps(D)}-\#\text{upward moving right-cusps(D)}.$$
In the first, writhe denotes the signed number of crossings in the front diagram $D$ representing $K$. For the reader's convenience, we indicate how to do this for the trefoil  in Figure \ref{fig:patterns2}.

\begin{figure}[!ht]
    \begin{center}
    		\subfigure[A Legendrian right-handed trefoil...]
		    {
		        \begin{picture}(160,215)(0,-40)
					\put(21,-20){writhe = $3$}
					\put(90,-20){$\#$ cusps = $6$}
					\put(15,-35){$\#\downarrow$ left = $2$}
					\put(85,-35){$\#\uparrow$ right = $1$}
					\includegraphics[bb=0 0 212 217,scale=.8]{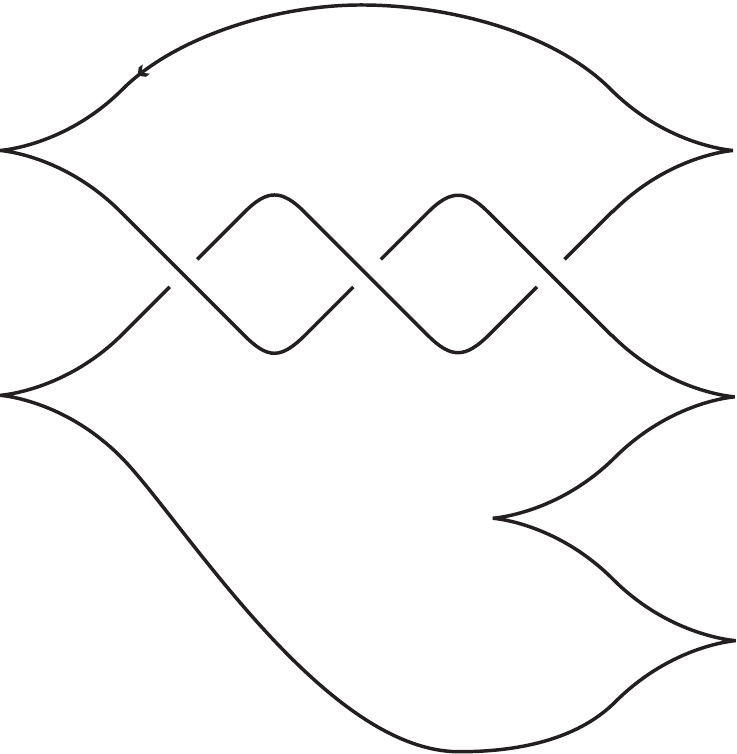}
					\put(0,90){$\to$}
			        \label{fig:legendrianRHT}
		        \end{picture}
		    }\hspace{.8cm}
		    \subfigure[...and its satellite with pattern $P$]
		    {
				\begin{picture}(160,215)(0,-40)
					\put(21,-20){writhe = $12$}
					\put(90,-20){$\#$ cusps = $20$}
					\put(15,-35){$\#\downarrow$ left = $5$}
					\put(85,-35){$\#\uparrow$ right = $4$}
			        \includegraphics[bb=0 0 212 217,scale=.8]{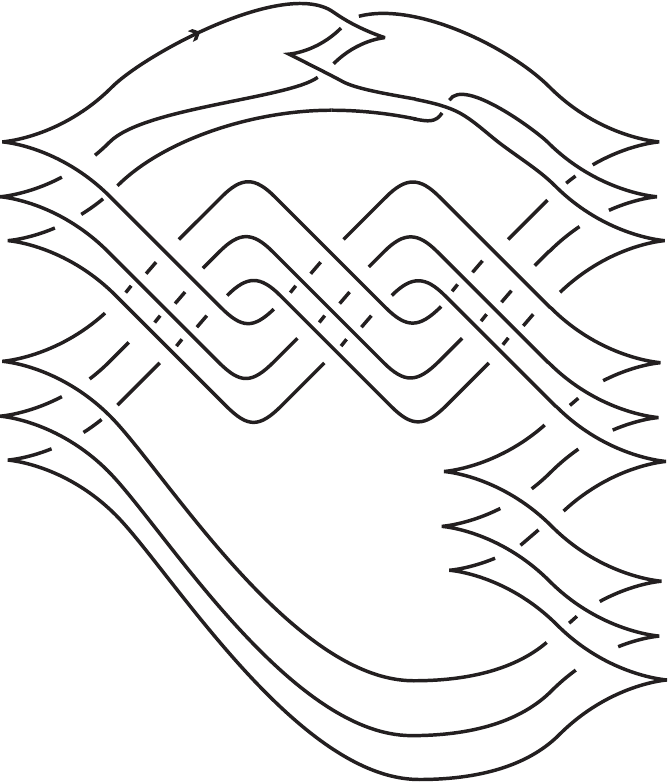}
			        \label{fig:legendriansatellite}
				\end{picture}
		    }
		\caption{}
	    \label{fig:patterns2}
    \end{center}
\end{figure}

The proof of Theorem \ref{thm:smooth} has an immediate corollary, which indicates that the Ozsv{\'a}th-Szab{\'o} and Rasmussen concordance invariants \cite{OzSz2,RasThesis,RasSlice}, $\tau(K)$ and $s(K)$, are not invariants of the smooth homology cobordism class of $M_K$.   

\begin{corollary}\label{cor:difftau} There exist topologically slice knots with homology cobordant (rel meridians) zero surgeries, but which have distinct $\tau$ and $s$ invariants.
\end{corollary}

\begin{proof}Both $\tau(K)$ and $s(K)$ refine the slice-Bennequin inequality in the following sense \cite{Olga2004,Olga2006, Shumakovitch}:
$$  tb(K)+|rot(K)|\le 2\tau(K)-1\le 2g_4(K)-1. $$
$$ tb(K)+|rot(K)|\le s(K)-1\le 2g_4(K)-1.$$ 
In Theorem \ref{thm:smooth}, we considered knots for which both inequalities are sharp;  $2\tau(K)=s(K)=2g_4(K)=2g(K).$  Applied to $P(K)$, the inequalities show that $\tau(P(K))>\tau(K)$ and $s(P(K))>s(K)$.
\end{proof}

It is tempting to believe that $K$ and $P(K)$ are topologically concordant, due to the fact that the Whitehead double embeds as a separating curve on the minimal genus Seifert surface of $P(K)$ (note that a simple exercise shows that $g(P(K))=g(K)+1$, with minimal genus Seifert surface for $P(K)$ provided by a minimal genus Seifert surface for $K$, glued to the twice punctured torus in $S^1\times D^2$ whose boundary is mapped to $P$ and a copy of the longitude $\partial(S^1\times D^2)$).  In light of this we can surger the Seifert surface along the Whitehead double, using the flatly embedded disk in $D^4$ provided by Freedman's theorem.  This only proves, however, that $g_4(K)=g_4(P(K))$.  A sufficient condition to  show that $K$ and $P(K)$ are concordant would be if $P\subset S^1\times D^2$ were concordant in $S^1\times D^2\times [0,1]$ to the core of the solid torus. This isn't true and can be seen by a variety of methods.  For instance,  if it were true then the knot obtained by adding a full twist to $P$ along the meridian of $S^1\times D^2$ would be topologically slice (perform $-1$ surgery along the meridian of the solid torus times $[0,1]$), a possibility ruled out by the classical signature \footnote{We thank Danny Ruberman for suggesting this argument, which is simpler than what we  initially used.}.

\section{Invariants of rational concordance}\label{sec:Qinvts}

In this section we will observe that many classical concordance invariants obstruct knots being $\Q$-concordant. This question has been considered in even greater generality in ~\cite{Cha2}. From this we deduce that only rarely is a knot $K$  $\Q$-concordant to its cable $K(p,1)$. We then observe that the $\tau$-invariant of Oszv\'{a}th-Szab\'{o} and Rasmussen can be used to obstruct smooth $\Q$-concordance (even between topologically slice knots). This is then used, in conjunction with known computations of $\tau$, to give examples of topologically slice knots $K$ which are not $\Q$-concordant to \emph{any} of their cables $K(p,1)$.

Before beginning we should point out that there \emph{do} exist non-slice knots $K$ for which $K$ is smoothly $\Q$-concordant to $K(p,1)$ for \emph{every} non-zero $p$. For suppose that $K$ is a smoothly $\Q$-slice knot that is not a smoothly slice knot. Examples of this are provided by the figure-eight knot (\cite[p.63]{Cha2}\cite[Lemma 2.2]{CHL6}), or more generally any (non-slice) strongly negative-amphicheiral knot ~\cite{Ka4}. Such a knot $K$ is $\Q$-concordant to the unknot $U$. Thus $K(p,1)$ is $\Q$-concordant to the unknot $U(p,1)$ for any $p$. Thus $K$ is $\Q$-concordant to $K(p,1)$.

\begin{proposition}\label{prop:coolexs} Suppose $K$ is a strongly negative-amphicheiral knot. Then for every non-zero $p$, $K$ is $\Q$-concordant to $K(p,1)$.
\end{proposition}

However, this is rare as the following subsection shows.

\begin{subsection}{Topological Invariants of rational concordance}

If $K$ is a knot in $S^3$ and $V$ is a Seifert matrix for $K$, then recall the Levine-Tristram $\omega$-signature of $K$, $\sigma_K(\omega)$, for any $|\omega|=1$ is the signature of
$$
(1-\omega)V+(1-\overline{\omega})V^T.
$$
The only explicit reference we found for the following theorem is ~\cite[Thm. 1.1]{ChaKo1} (see also ~\cite[Theorem 1.7]{CO1}\cite{CO1a}).

\begin{proposition}\cite[Thm. 1.1]{ChaKo1}\label{prop:sigsinvt} If $K_0$ is CAT $\Q$-concordant to $K_1$ then for any prime $p$ and $\omega=e^{2\pi ik/p}$,
$$
\sigma_{K_0}(\omega)=\sigma_{K_1}(\omega).
$$
\end{proposition}

Indeed Cha shows that the ordinary algebraic knot concordance group embeds into an algebraic $\Z_{(2)}$-concordance group ~\cite[Section 2.2, 4.4]{Cha2} ( see also ~\cite[Theorem 1.7]{CO1}\cite{CO1a}).

\noindent Note that, since such values of $\omega$ are dense in the circle and since the Levine signature functions of $K_0$ and $K_1$ are constant except at roots of the Alexander polynomial, Proposition~\ref{prop:sigsinvt} implies that these functions agree for $K_0$ and $K_1$, except possibly at roots of their respective Alexander polynomials.

\begin{corollary}\label{cor:finiteorder} If $K$ is CAT $\Q$-concordant to $K(p,1)$ for some $p>1$, then $K$ is of finite order in the algebraic knot concordance group.
\end{corollary}
\begin{proof} [Proof of Corollary~\ref{cor:finiteorder}] Assume that $K$ is TOP $\Q$-concordant to $K(p,1)$. Then their signatures agree (except possibly at roots of the Alexander polynomials). Suppose \emph{some} Levine-Tristram-signature of $K$ were non-zero. Since the $\omega=1$ signature always vanishes and since the Levine-Tristram signature function is locally constant, possibly jumping only at roots of the Alexander polynomial, we can  choose $\omega$ so that $\omega^p$ is (or is close to) the ``first'' value on the unit circle (smallest argument) for which $K$ has non-zero signature (and avoiding roots of the Alexander polynomial). That is we can choose $\omega$ so that
$$
\sigma_{K}(\omega)=0, ~~\text{and}~~ \sigma_{K}(\omega^p)\neq 0.
$$
But it is known by ~\cite{Kea2,Lith2} that,
$$
\sigma_{K(p,1)}(\omega)=\sigma_K(\omega^p).
$$
Combining this with Proposition~\ref{prop:sigsinvt}, we see that
$$
\sigma_K(\omega^p)=\sigma_K(\omega).
$$
This is false for our particular choice of $\omega$ above. Hence the signature function of $K$ vanishes (excluding roots of the Alexander polynomial). It is known that this is equivalent to $K$ being of finite order in the algebraic knot concordance group ~\cite{L5}.
\end{proof}

\begin{corollary}\label{cor:trefoil1} The right-handed trefoil knot, $T$, is not TOP $\Q$-concordant to $T(p,1)$ for any $p>1$.
\end{corollary}

There are other related papers that discuss the question of whether a knot is $\Z$-concordant to its $(p,1)$-cable ~\cite{Ka5}\cite{Lith1}\cite{CRL}.

The following criteria can be applied even when the knot signatures fail. The first does not seem to appear in the literature although it does follow, for example, from combining results of the much more general \cite{Cha2}. The second is implicit in \cite{Cha2}. We sketch a proof in order to make a pedagogical point about rational concordance.

\begin{proposition}\label{prop:polysinvt} If $K_0$ is CAT $\Q$-concordant to $K_1$ then for some positive integer $k$ and for some integral polynomial  $f$,
$$
\delta_{0}\left(t^k\right)\delta_{1}\left(t^k\right)\doteq f(t)f\left(t^{-1}\right),
$$
where $\delta_{i}(t)$ is the Alexander polynomial of $K_i$.
\end{proposition}

Most generally, let $\mathcal{B}\ell^K(t)$ denote the nonsingular Blanchfield linking form defined on the rational Alexander module of $K$, $\mathcal{A}^\Q(K)\equiv H_1(M_K;\Q[t,t^{-1}])$. Then,

\begin{proposition}\label{prop:Blanchinvt} If $K_0$ is CAT $\Q$-concordant to $K_1$ then for some positive integer $k$,
$$
\mathcal{B}\ell^{K_0}(t^k)\sim \mathcal{B}\ell^{K_1}(t^k),
$$
where these denote the induced forms on the module
$$
\mathcal{A}^\Q(K_i)\otimes_{\Q[t,t^{-1}]}\Q[t,t^{-1}]
$$
where here the right-hand $\Q[t,t^{-1}]$ is a module over itself via the map $t\to t^k$; and $\sim$ denotes equality in the Witt group of such forms (see \cite{Cha2}\cite{Hi}).
\end{proposition}

\begin{proof} Suppose $K_0$ is CAT $\Q$-concordant to $K_1$ via an annulus, $A$, embedded in a $\Q$-homology $S^3\times [0,1]$, $W$. Let $E_0$, $E_1$ and $E_A$ denote the exteriors of $K_0$, $K_1$ and $A$ respectively. Then
$$
\frac{H_1(E_*;\Z)}{\mathrm{torsion}}\cong \Z
$$
for $*=0,1,A$. The \emph{complexity of the concordance}  is the positive integer $k$ for which the image of the meridian $\mu_i$, for $i=0,1$,   under the inclusion-induced map $j_i$
$$
\frac{H_1(E_0;\Z)}{\mathrm{torsion}}\overset{j_0}{\longrightarrow}\frac{H_1(E_A;\Z)}{\mathrm{torsion}}
\overset{j_1}{\longleftarrow}\frac{H_1(E_1;\Z)}{\mathrm{torsion}},
$$
is $\pm k$ times a generator. This was defined in \cite{CO1,CO1a}, but see also \cite{Cha2, ChaKo1}, and was called the \emph{multiplicity} in \cite[page 463]{COT}. There is a unique epimorphism
$$
\phi:\pi_1(E_A)\to \Z.
$$
This defines a coefficient system on $E_A$ and also on $E_i$ for $i=0,1$  by setting $\phi_i=\phi\circ j_i$. Then it is well-known that the Alexander modules using these induced coefficient systems are not the ordinary Alexander modules but rather,
$$
H_1(E_i;\Z[t,t^{-1}])\cong \mathcal{A}(K_i)\otimes_{\Z[t,t^{-1}]}\Z[t,t^{-1}],
$$
where the right-hand $\Z[t,t^{-1}]$ is a module over itself via the map $t\to t^k$. The order of such a module is well-known to be $\delta_{i}(t^k)$ where $\delta_{i}(t)$ is the order of $\mathcal{A}(K_i)$. (This ``tensored up''  module is the same as the Alexander module of the $(k,1)$-cable of $K_i$). The coefficient system $\phi$ also induces Blanchfield linking forms on these modules and these differ from the ordinary Blanchfield form in the analogous manner.

If $A$ were an actual concordance then we have the classical result that the kernel, $P$, of the map
$$
\mathcal{M}\equiv H_1(E_0;\Q[t,t^{-1}])\oplus H_1(E_1;\Q[t,t^{-1}])\to H_1(E_A;\Q[t,t^{-1}])
$$
is self-annihilating with respect to the ordinary Blanchfield forms. It would then follow (by definition) that the Blanchfield forms are equivalent in the Witt group. It also would follow that $P$ is isomorphic to the dual of $\mathcal{M}/P$, quickly yielding the classical result 
$$
\delta_{K_0}\left(t\right)\delta_{K_1}\left(t\right)\doteq f(t)f\left(t^{-1}\right)
$$
for some polynomial $f$. In the situation that $A$ is only a $\Q$-concordance, these results are also known (see for example \cite[Theorem 4.4, Lemma 2.14]{COT}). The only difference is that the relevant modules and forms are not the ordinary Alexander modules but rather are  ``tensored up'' as above, and the orders of the relevant modules are not the actual Alexander polynomials of $K_i$, but are $\delta_i(t^k)$. The claimed results follow.
\end{proof}

\begin{corollary}\label{cor:twist} Suppose $K$ is the $3$-twist knot with a negative clasp . Then, although $K$ is of finite order in the algebraic concordance group, $K$ is not TOP $\Q$-concordant to $K(p,1)$ for any $p>1$.
\end{corollary}
\begin{proof} The Alexander polynomial of $K$ is $\delta(t)=3t-7+3t^{-1}$, whereas the Alexander polynomial of $K(p,1)$ is $\delta(t^p)$. If $K$ were TOP $\Q$-concordant to $K(p,1)$  then by Proposition~\ref{prop:polysinvt}, for some positive $k$ and integral polynomial $f(t)$,
$$
\delta\left(t^k\right)\delta\left(t^{kp}\right)=\pm t^g f(t)f\left(t^{-1}\right).
$$
But $\delta\left(t^k\right)$ (and thus ($\delta\left(t^{kp}\right)$) is irreducible for any $k$ ~\cite[Prop. 3.18]{Cha2}. This  contradicts unique factorization  if $p>1$.
\end{proof}

Casson-Gordon invariants and higher-order von-Neumann signatures should yield higher-order obstructions to $\Q$-concordance.

\end{subsection}
\begin{subsection}{Smooth Rational concordance invariants for topologically slice knots}

The Ozsv\'{a}th-Szab\'{o}-Rasmussen $\tau$-invariant is an integral-valued knot invariant that is invariant under smooth concordance and additive under connected sum ~\cite{OzSz2}. It is not invariant under topological concordance and therefore may be used in cases where algebraic invariants fail. It is also known  that it is an invariant of smooth rational concordance.

\begin{proposition}\label{prop:tau1} If $K$ is smoothly $R$-concordant to $J$ then $\tau(K)=\tau(J)$.
\end{proposition}
\begin{proof} We are given that $K$ and $J$ are connected by a smooth annulus $A$ in a smooth $R$-homology $S^3\times [0,1]$, $W$. Choose an arc in $A$ from $K$ to $J$. By deleting a small neighborhood of this arc from $W$ we arrive at a smooth $R$-homology $4$-ball $\mathcal{B}$. The annulus $A$ is cut open yielding a $2$-disk whose boundary is the knot type of $K\#-J$. Thus $K\#-J$ is smoothly $R$-slice.  By ~\cite[Theorem 1.1]{OzSz2}, $\tau(K\#-J)=0$, so $\tau(K)=-\tau(-J)=\tau(J)$, the last property being also established in ~\cite{OzSz2}. 
\end{proof}

It is a fascinating question whether or not the $s$-invariant is zero on $\Q$-slice knots.

\end{subsection}

\section{Question~\ref{question:KirbyQ}}\label{sec:mainthm}

\begin{theorem}\label{thm:mainapplication} The answer to Question~\ref{question:KirbyQ} is ``No,'' in both the smooth and topological category. In the smooth category there exist counterexamples that are topologically slice. 
\end{theorem}
\begin{proof}[Proof of Theorem~\ref{thm:mainapplication}] Let $T$ be the trefoil knot, or indeed any knot with some non-zero Levine-Tristram signature. By Corollary~\ref{cor:cable}, for any $p>1$, $M_T$ is smoothly $\Q$-homology cobordant to $M_{T(p,1)}$. But  by Corollary~\ref{cor:trefoil1}, or more generally by Corollary~\ref{cor:finiteorder}, $T$ is not TOP $\Q$-concordant to $T(p,1)$. Therefore the answer to Question~\ref{question:KirbyQ} is ``No'' in either category.  Proposition~\ref{prop:polysinvt} can also be used to give examples that have finite order in the algebraic knot concordance group.

We claim that in the smooth category there exist topologically slice examples. These are provided  by combining Corollary~\ref{cor:difftau} and Proposition~\ref{prop:tau1}. Different pairs of examples can be obtained as follows. Let $K_0$ be the untwisted, positively-clasped Whitehead double of the right-handed trefoil knot, and let $K_1$ be the $(p,1)$-cable of $K_0$.  The Alexander polynomials of $K_0$ and $K_1$ are  equal to $1$, and so work of Freedman~\cite{FQ} implies these knots are topologically slice. The zero-framed surgeries on these knots are smoothly $\mathbb{Q}$-homology cobordant rel meridians by Corollary~\ref{cor:cable}.  If these knots were smoothly $\mathbb{Q}$-concordant then by Proposition~\ref{prop:tau1}, $\tau(K_0)=\tau(K_1)$. But this is not true. In ~\cite{He1} it is shown that $\tau(K)=1$ whereas  ~\cite[Theorem 1.2]{He2}  shows that $\tau(K(p,1))=p\,\tau(K)=p$.  Thus if $p\neq 1$, these knots are not smoothly $\mathbb{Q}$-concordant. 
\end{proof}

\bibliographystyle{amsplain}
\bibliography{QConcBib}

\end{document}